\documentclass[12pt,a4paper]{article}

\usepackage{graphicx}
\usepackage{tikz}
\usetikzlibrary{decorations.fractals}

\usepackage{graphicx,color}

\usepackage{amsmath,amsthm}
\usepackage{amssymb,latexsym}
\usepackage{enumerate}
\usepackage{amscd,amssymb,amsthm}
\usepackage{graphicx}

\usepackage{mathrsfs,amssymb}

\newtheorem{theorem}{Theorem}[section]

\newtheorem{corollary}[theorem]{Corollary}
\newtheorem{lemma}[theorem]{Lemma}

\newtheorem{definition}{Definition}[section]
\theoremstyle{definition}
\newtheorem{remark}[theorem]{Remark}

\numberwithin{equation}{section}

\newcommand{\wred}{\textcolor{black}}

\newcommand{\leb}{{\mathscr L}}
\newcommand{\ssob}{\mathscr W}
\newcommand{\sob}{\mathscr W^1\leb^p\log^{-\alpha}\leb}
\newcommand{\Sob}[1]{\mathscr W^1\leb^p\log^{-#1}\leb}
\newcommand{\Rz}{\mathbb{R}}
\newcommand{\E}{{\fam 0 e}}
\DeclareMathOperator{\divergenza}{div}

\newcommand{\naturale}{{\mathbb N}}
\newcommand{\ee}{\varepsilon}
\newcommand{\ds}{\displaystyle}

\newcommand{\dmedint}{{}\hbox
{\vrule height 2,95pt depth -2,2pt width 6pt}\kern-0.94em }
\newcommand{\tmedint}{{}\hbox
{\vrule height 2,7pt depth -2,3pt width 5pt}\kern-8,5pt }
\newcommand{\smedint}{{}\hbox
{\vrule height 2,1pt depth -1,7 pt width 3pt}\kern-6,2pt }
\newcommand{\ssmedint}{{}\hbox
{\vrule height 1,7pt depth -1,3 pt width 3pt}\kern-6,2pt }
\newcommand{\medint}{\,{\mathchoice
{\dmedint}{\tmedint}{\smedint}{\ssmedint}}\int}
\newcommand{\implica}{\quad\Longrightarrow\quad}
\newcommand{\plapl}[2]{  {\rm div} \left(  \left| \nabla #1 \right|^{#2-2} \nabla #1 \right)  }
\newcommand{\aplapl}[3]
{  {\rm div} \left(  \left\langle 
 #1 \nabla #2, \nabla #2
\right\rangle^{ \frac {#3 - 2} 2  }
 #1 \nabla #2
\right)  }

\newcommand{\ellmatr}{A}
\newcommand{\scal}[2]{\left \langle #1 , #2 \right \rangle}

\newcommand{\RR}{\mathbb R}

\newcommand{\grandLebesgue}[2]{\leb^{#1)}\left(#2\right)}  

\def\Xint#1{\mathchoice
{\XXint\displaystyle\textstyle{#1}}%
{\XXint\textstyle\scriptstyle{#1}}%
{\XXint\scriptstyle
\scriptscriptstyle{#1}}%
{\XXint\scriptscriptstyle
\scriptscriptstyle{#1}}%
\!\int}
\def\XXint#1#2#3{{
\setbox0=\hbox{$#1{#2#3}{\int}$}
\vcenter{\hbox{$#2#3$}}\kern-.5\wd0}}
\def\dashint{\Xint-}

\newcommand{\average}[3]{ \dashint_{#1} #2 \, d#3} 
\newcommand{\averagebis}[1]{ \dashint_{#1} } 

\begin{document}

\title{Estimates for $p$-Laplace\\type equation in a limit case}

\date{\today}


\author{Fernando Farroni\and Luigi Greco\and Gioconda Moscariello  
\thanks{This   research has been  supported by 
the 2008 ERC Advanced Grant 226234 ``Analytic Techniques for
Geometric and Functional Inequalities''
and by the 2010 PRIN ``Calculus of Variations'.}
}


\maketitle

\begin{abstract}
We
study 
the
Dirichlet problem
for a 
$p$--Laplacian type  operator   
in the setting of 
the Orlicz--Zygmund space $\leb^q\log^{-\alpha}\leb(\Omega,\mathbb R^n)$,
$q >1$ and $\alpha>0$.
More precisely, our aim is to establish 
which assuptions on the parameter $\alpha>0$  lead to
existence, uniqueness of the solution and  continuity of the associated nonlinear operator.
\\
\newline
\\
\noindent  \textit{Keywords}:
Dirichlet problems, $p$--Laplace operators, existence, uniqueness, continuity, Orlicz--Sobolev spaces.
\\
\noindent  \textit{Mathematics Subject Classification} (2000): 35J60
\end{abstract}

\section{Introduction}
Let $\Omega$ be a bounded Lipschitz domain of $\mathbb R^N$, $N\geqslant 2$.
We consider the Dirichlet problem
\begin{equation}\label{A harm intro f}
\left\{
\begin{array}{rl}
&  
\divergenza \mathcal A (x,\nabla u)   = \divergenza f  	\qquad\text{in $  \Omega$},
  \\
\, \\
&   u = 0   
\qquad\text{on $\partial \Omega$}
, \\
\end{array}
\right.
\end{equation}
where $\mathcal A\colon\Omega \times \mathbb R^N \rightarrow \mathbb R^N$
is a Carath\'eodory vector field satisfying the following assumptions for a.e. $x\in \Omega$ and all $\xi,\eta \in \mathbb R^N$
\begin{align}
\label{1.2}
&
\left\langle
\mathcal A (x,\xi) , \xi   \right\rangle \geqslant a|\xi|^p
\\
\label{1.3}
&
|\mathcal A (x,\xi) - \mathcal A (x,\eta)| \leqslant b |\xi-\eta| \left( |\xi|+|\eta| \right)^{p-2}
\\
\label{1.4}
&
\left\langle
\mathcal A (x,\xi) - \mathcal A (x,\eta) , \xi-\eta   \right\rangle \geqslant  \wred{a} 
|\xi-\eta|^2 \left( |\xi|+|\eta| \right)^{p-2}
\end{align}
where 
\wred{$p \geqslant 2$, 
%
%
$0 < a \leqslant b$}. 

\medskip

Let $f=\left(f^1,f^2,\dots,f^N\right)$ be a vector field of class $\leb^s\left(\Omega,\mathbb R^N\right)$, $1\leqslant s \leqslant q$ where $q$ is the conjugate exponent of $p$, i.e. $pq=p+q$.
\begin{definition}\label{Def weak}
A function $u \in  W^ {   1 , r } _0 (\Omega)$, $p-1 \leqslant r \leqslant p$, 
is
a solution of
\eqref{A harm intro f}
if
\begin{equation}\label{weak}
\int_\Omega
\scal{\mathcal A  (x,\nabla u)  }{\nabla \varphi} 
dx
=
\int_\Omega
\scal f {\nabla \varphi} 
dx,
\end{equation}
for every $\varphi \in C^\infty_0(\Omega)$. 
\end{definition}

By a routine argument, it can be seen 
that the identity \eqref{weak} 
still holds
for functions $\varphi \in \ssob^{1, \frac {r} {r-p+1} }(\Omega)$
with compact support.
We shall refer to such
a
solution
as a distributional solution
or (as some people say)
as a very weak solution
\cite{ISLincei,Lewis}.

\medskip

We point out that, if $r<p$,
such a solution may have infinite energy, i.e. $|\nabla u|\not\in \leb^p(\Omega)$.
The
existence of a solution $u \in \ssob_0^{1,p-1}(\Omega)$ 
to problem  \eqref{A harm intro f} is obtained in \cite{BG} 
when $\divergenza f$ belongs to $\leb^1 \left(\Omega,\mathbb R^N\right)$.
It is well known that the uniqueness of solutions to \eqref{A harm intro f}
in the sense of Definition \ref{Def weak}
generally fails \cite{Serrin,AIM}.
Then, 
other possible definitions have been introduced, as 
the so--called
duality solutions   \cite{Stampacchia},  
the
approximation solutions (SOLA) \cite{BG},
the entropy solutions \cite{Murat,KL,BOG}.
%
%
Recent results for the regularity of such solutions 
are given       in     
\cite{Mingione1,Mingione2}.
However,
these ideas do not apply if one wants
to investigate
the uniqueness of a distributional solution.
At the present time the problem remains unclear, 
unless for $p=2$ \cite{B,FS}
and
$ p=N$ \cite{DHM}.
In the case $p=2$  
%
%
\wred
{
the range of exponents $r$
allowing for a comprehensive theory is known,
see
\cite{AIS,Iwaniec-Sbordone Ann Poinc}.
}
In the general case,
uniqueness is proved in the setting of
the
grand Sobolev space  
(see \cite{GIS}).

\medskip

Our  goal in
the present paper is to  study problem \eqref{A harm intro f} assuming that the datum $f$ lies in
the Orlicz--Zygmund space $\leb^q\log^{-\alpha}\leb(\Omega,\mathbb R^n)$, $\alpha>0$.
More precisely, our aim is to establish under which assuptions on the parameter $\alpha>0$
we can define a 
continuous operator
\begin{equation} \label{H}
\mathcal H : \, \leb^q\log^{-\alpha}\leb(\Omega,\mathbb R^n) \rightarrow \leb^p\log^{-\alpha}\leb(\Omega,\mathbb R^n)
\end{equation}
which carries a given vector field $f$ into the gradient field $\nabla u$.

\medskip

In the case $\alpha \leqslant 0$, in the literature there are
several
results on the continuity of the operator defined in \eqref{H}  
\cite{M Nodea,DM,IO}.
Moreover,
as a consequence of the results in \cite{FS}   and \cite{B}
and the interpolation theorem of 
\cite{BR},
when $p=2$ the operator $\mathcal H$ is 
Lipschitz
continuous for any $-\infty<\alpha<\infty$.
Actually, for $p=2$ 
and suitable
$\alpha >0$,
the existence 
for
problem \eqref{A harm intro f} is also ensured for not uniformly elliptic equations
\cite{MPP}.

\medskip

Here we consider the case $p>2$.
Our main results are the following.

\begin{theorem}\label{main main}
For each $f \in \leb^q\log^{-\alpha}\leb\left(\Omega,\mathbb R^N\right)$, $1<q<2$ and $0 < \alpha \leqslant \frac p {p-2}$, the problem \eqref{A harm intro f} admits a unique solution $u\colon   \Omega \rightarrow \mathbb R$, such that $\nabla u \in  \leb^p\log^{-\alpha}\leb(\Omega,\mathbb R^n)$. 
\wred
{There exists a constant $C>0$ depending on $n, p, \alpha, a$ and $b$ such that 
the following estimates
holds true
\begin{equation}\label{tag eqn}
\left
\|
\nabla u
\right\|^p_{\leb ^ p \log ^{-\alpha} \leb}
\le 
C
\left
\|
f
\right\|^q_{\leb ^ q \log ^{-\alpha} \leb}
\end{equation}
}
Moreover
the operator $\mathcal H$ is continuous. 
\end{theorem}

\begin{theorem}\label{main main 2intro}
\wred
{
There exists a constant $C>0$ depending on $n, p, \alpha, a$ and $b$ such that, if
}
%
%
$f$ and $g$ belong to   ${\leb ^ q \log ^ {-\alpha} \leb (\Omega,\mathbb R^n)}$,
$1<q<2$ and $0 < \alpha < \frac p {p-2}$, 
then
\begin{equation}\label{3.3 ter}
\|
\mathcal H f - \mathcal H g
\|^p_{\leb^p \log^{-\alpha}\leb}
\leqslant
C
\left(
\|
f-g
\|_{\leb^q \log^{-\alpha}\leb}^{q(1-\gamma)}
\left\|
|f|+|g|
\right\|^{q \gamma }_{\leb^q \log^{-\alpha}\leb}
\right),
\end{equation}
where $\gamma =   \alpha \frac {p-2} p$. 
%
%
\end{theorem}

\medskip

We point out that  Theorem \ref{main main} 
improves the result of \cite{GIS} in two different directions.
First of all, when $0 < \alpha < \frac p {p-2}$, it gives higher integrability of the solutions found in \cite{GIS}. On the other hand, the case $\alpha=\frac p {p-2}$ is not covered by \cite{GIS}.

\medskip

In the particular case that the vector field $\mathcal A$ takes the form
\begin{equation}\label{matric matcal A}
\mathcal A (x,\xi)
=
\left\langle
A(x)\xi,\xi
\right\rangle^{\frac {p-2} 2}
A(x) \xi
\end{equation}
where $A\colon\Omega \rightarrow \mathbb R^{N\times N}$ is a measurable, symmetric, uniformly elliptic matrix field,
we also prove a stability theorem for solutions to problem \eqref{A harm intro f} in terms of the characteristic of $A$ (see Section \ref{sect4}).
The
\emph{characteristic} of the symmetric  matrix field $\ellmatr\colon \Omega \rightarrow
\RR^{N\times N}$
(see \cite{Iwaniec})
is defined
as the quantity
\begin{equation}\label{char}
K_\ellmatr = \underset  {x \in \Omega} { {\rm ess \,  sup} } 
\left(
1+|\ellmatr(x)-I|
\right)^{\frac p 2}.
\end{equation} 
Observe that $K_\ellmatr \ge 1$ and $K_A =1$ if and only if $A$ is the identity matrix.

\begin{theorem}
\label{comparison}
Assume that
$A\colon \Omega \rightarrow
\RR^{N\times N}$
%
%
is 
a measurable  
symmetric
matrix field satisfying
the 
ellipticity 
bounds
\begin{equation}\label{unif ell}
a^{\frac 2 p}
|\xi|^2
\le
\scal{A(x) \xi}{ \xi }
\le
b^{\frac 2 p}
|\xi|^2, 
\end{equation}
for a.e. $x\in \Omega$, for every $\xi\in \mathbb R^N$.
\wred
{
There exists a constant $C>0$ depending on $n, p, \alpha, a$ and $b$ such that, if
$u , v \in \Sob  \alpha (\Omega)$, with
$0 < \alpha < \frac p {p-2}$, 
verify
}
\begin{equation}\label{A harm intro 2}
\left\{
\begin{array}{rl}
&  
\aplapl {\ellmatr(x)} u p = 
\plapl v p	\qquad\text{in $  \Omega$},
  \\
\, \\
&   u = v   
\qquad\text{on $\partial \Omega$}
, \\
\end{array}
\right.
\end{equation}
\wred{then
\begin{equation}\label{4.7}
\begin{split}
\|\nabla u - \nabla v \|^p_{\leb^p\log^{-\alpha}\leb}
\leqslant
C \left(K_A-1\right)^ { q (1-\gamma) }   
K_A^{ q (\gamma + 1) }
\left\| |\nabla u| + |\nabla v| \right\|^p_{\leb^p\log^{-\alpha}\leb}
\end{split}
\end{equation}
%
where $\gamma =   \alpha \frac {p-2} p$. 
}
\end{theorem}

The main tool to prove our results is the Hodge decomposition and fine properties of the norm in the Zygmund spaces developped in Section \ref{sect2}.

\section{Preliminary results}\label{sect2}
\subsection{Basic notation}

%
We indicate that quantities $a,b\ge0$ are equivalent  by writing
$a \sim b$; namely,     $a \sim b$ will mean that
there exist 
\wred{constants}
 $c_1,c_2  >  0$
such that $c_1 a \le b \le c_2 a$.
Similarly,   $a \lesssim b$  ($a \gtrsim   b$ respectively)
will mean that there exists $c >  0$ such that 
$
a \le c b
$
($a \ge c   b$ respectively).
%
%
%
%

\medskip

\noindent From now on, 
$\Omega$ will denote  a bounded \wred{Lipschitz} domain in $\mathbb R^N$. 
For a function $v \in \leb^p(\Omega)$ with $1\le p<\infty$ 
we set 
\[
\|
v
\|_p
=
\left(
\average \Omega {|v|^{p}} x
\right)^{\frac 1 {p}}
\]
Barred integrals denote averages, namely $\averagebis \Omega {}{}=\frac 1 {|\Omega|}\int_\Omega$.
%
%
%
%
%
\subsection{Marcinkiewicz Spaces}
For \wred{$0 < p<\infty$},  
the
\emph{Marcinkiewicz space weak-$\leb^p (\Omega)$
},
\wred{
also denoted by 
$\leb^{p,\infty} (\Omega) $,
consists} 
of
all measurable  functions 
$g:\Omega\rightarrow \mathbb R$
such that
$$
\left\|g\right\|^p_{  \leb ^    {p,\infty}  (\Omega)}
\equiv
\left\|g\right\|^p_{p,\infty}=\sup_{t>0}t^{ p}| \{x\in \Omega:|g(x)|>t\} |< \infty.
$$
A useful property of the Marcinkiewicz norm is given by the following identities
\begin{equation}\label{nota}
\left\|  |g|^{\alpha}  \right\|^p_{ p,\infty }=\|g\|^{\alpha p}_{ \alpha p,\infty } 
\qquad
\text{for $\alpha>0$}.
\end{equation}
For $1<q<p$ one has
\begin{equation*}
\leb ^{p,\infty }(\Omega) \subset \leb^q (\Omega).
\end{equation*}
\wred{We shall
appeal} to the following H\"older type inequality
\begin{equation}\label{Holder Mar}
\| v \|_{\leb^q  (E)}
\leqslant
\left(
\frac
p
{p-q}
\right)
^{\frac 1 q}
|E|^{-\frac 1 q}
|\Omega|^{\frac 1 q - \frac 1 p}
\| v \|_{L^{p,\infty}  (\Omega)}
\end{equation}
which holds true for $v \in \leb^{p,\infty}(\Omega)$, $E\subset\Omega$ and $q<p$.
\subsection{Grand Lebesgue and  grand Sobolev Spaces} \label{grand Sob sec}
For $1< p<\infty$ 
we denote by $\grandLebesgue p \Omega$
\emph{the grand--Lebesgue space} $\grandLebesgue{p}{\Omega}$
consisting of all functions 
$v \in \bigcap _{ 0 < \varepsilon \le p-1}
L^{ p-\varepsilon} (\Omega)$
such that
\begin{equation}
\|
v
\|_{p)}
=
\sup_{ 0 < \varepsilon \le p-1}
\varepsilon ^{\frac 1 p}
\left(
\average \Omega {|v|^{p-\varepsilon}} x
\right)^{\frac 1 {p-\varepsilon}}
<\infty. 
\end{equation} 
Moreover
\begin{equation}
\|
v
\|_{p)} \sim 
\sup_{ 0 < \varepsilon \le p-1}
\left(
\varepsilon
\average \Omega {|v|^{p-\varepsilon}} x
\right)^{\frac 1 {p-\varepsilon}}
.
\end{equation}
The Marcinkiewicz class ${\rm weak}-\leb^p(\Omega)$
is contained in $\leb^{p)}(\Omega)$
(see \cite[Lemma 1.1]{IS}). 

\medskip

\wred{More generally}, if
$\alpha >0$
we denote by $\leb^{\alpha,p)}(\Omega)$
\emph{the grand--Lebesgue space}  
consisting of all functions 
$v \in \bigcap _{ 0 < \varepsilon \le p-1}
\leb^{ p-\varepsilon} (\Omega)$
such that
\begin{equation}
\|
v
\|_{\alpha,p)}
=
\sup_{ 0 < \varepsilon \le p-1}
\varepsilon ^{\frac \alpha p}
\left(
\average \Omega {|v|^{p-\varepsilon}} x
\right)^{\frac 1 {p-\varepsilon}}
<\infty. 
\end{equation}

\subsection{Zygmund spaces}
We shall need to consider the Zygmund space $\leb^q\log^{-\alpha}\leb(\Omega)$, for $1<q<\infty$, $\alpha>0$. 
This is the Orlicz space generated by the function
\[\Phi(t)=t^q \log^{-\alpha}(a+t)\,,\qquad t\geqslant 0\,,\]
where $a\geqslant\E$ is a suitably large constant, so that $\Phi$ is increasing and convex on $[0,\infty[$. The choice of $a$ will be immaterial. 
More explicitly, for a measurable function $f$ on $\Omega$, $f\in\leb^q\log^{-\alpha}\leb(\Omega)$ simply means that
\[\int_\Omega |f|^q\log^{-\alpha}(a+|f|)\,dx<\infty\,.\]
It is customary to consider the Luxemburg norm
\[[f]_{\leb^q\log^{-\alpha}\leb}=\inf\left\{\lambda>0\ :\ \medint_\Omega\Phi(|f|/\lambda)\,dx\leqslant 1\right\},\]
and $\leb^q\log^{-\alpha}\leb(\Omega)$ is a Banach space. However, we shall introduce an equivalent norm, which involves the norms in $\leb^{q-\ee}(\Omega)$, for $0<\ee\leqslant q-1$, and is more suitable for our purposes. For $f$ measurable on $\Omega$, we set
\begin{equation}\label{norma equiv}
\|f\|_{\leb^q\log^{-\alpha}\leb}=\left\{\int_0^{\ee_0}\ee^{\alpha-1}\|f\|_{q-\ee}^q\,d\ee\right\}^{1/q}
\end{equation}
Here $\ee_0\in{}]0,q-1]$ is fixed. The following is a refinement of a result of \cite{Gr}.
\begin{lemma}\label{nuova norma}
We have $f\in \leb^q\log^{-\alpha}\leb(\Omega)$ if and only if
\begin{equation}\label{finitezza}
\|f\|_{\leb^q\log^{-\alpha}\leb}<\infty\,.
\end{equation}
Moreover, $\|~\|_{\leb^q\log^{-\alpha}\leb}$ is a norm equivalent to the Luxemburg one, that is, there exist constants $C_{i}=C_{i}(q,\alpha,a,\ee_0)$, $i=1,2$, such that for all $f\in \leb^q\log^{-\alpha}\leb(\Omega)$
\[C_{1}\,[f]_{\leb^q\log^{-\alpha}\leb}\leqslant \|f\|_{\leb^q\log^{-\alpha}\leb}\leqslant C_{2}\,[f]_{\leb^q\log^{-\alpha}\leb}\,.\]
\end{lemma}
\begin{proof}
It is easy to check that $\|~\|_{\leb^q\log^{-\alpha}\leb}$ defined by \eqref{norma equiv} is a norm.\par
Let $f$ be  a measurable function defined in $\Omega$. We clearly have
\[|f|^q(a+|f|)^{-\ee}\leqslant |f|^{q-\ee}\leqslant 2^{q-1}[a^{q}+|f|^{q}(a+|f|)^{-\ee}]\,,\]
for a.e.\ in $\Omega$, hence integrating
\[\medint_\Omega |f|^q(a+|f|)^{-\ee}\,dx\leqslant \|f\|_{q-\ee}^{q-\ee}\leqslant 2^{q-1}a^{q}+2^{q-1}\medint_\Omega |f|^{q}(a+|f|)^{-\ee}\,dx\,.\]
This in turn implies
\begin{equation}\label{201309111}
\begin{array}{rl}
\ds \int_0^{\ee_0}\ee^{\alpha-1}\kern-.7em &\ds\left[\medint_\Omega |f|^q(a+|f|)^{-\ee}\,dx\right]d\ee\leqslant\int_0^{\ee_0}\ee^{\alpha-1}\|f\|_{q-\ee}^{q-\ee}\,d\ee\\\\
&\ds{}\leqslant2^{q-1}a^{q}\frac{\ee_0^\alpha}\alpha+2^{q-1}\int_0^{\ee_0}\ee^{\alpha-1}\left[\medint_\Omega |f|^q(a+|f|)^{-\ee}\,dx\right]d\ee\,.
\end{array}
\end{equation}
Moreover,
\[\int_0^{\ee_0}\ee^{\alpha-1}(a+|f|)^{-\ee}\,d\ee=\log^{-\alpha}(a+|f|)\int_0^{\ee_{0}\log(a+|f|)}\tau^{\alpha-1}\E^{-\tau}\,d\tau\]
and
\[
\int_0^{\ee_{0}\log a}\tau^{\alpha-1}\E^{-\tau}\,d\tau \leqslant \int_0^{\ee_{0}\log(a+|f|)}\tau^{\alpha-1}\E^{-\tau}\,d\tau \leqslant \int_0^{\infty}\tau^{\alpha-1}\E^{-\tau}\,d\tau
\]
Therefore from \eqref{201309111} we get
\begin{equation}\label{201309112}
\begin{array}{rl}
\ds C_3\medint_\Omega |f|^q\log^{-\alpha}(a+|f|)\,dx\kern-.7em&\ds{}\leqslant \int_0^{\ee_0}\ee^{\alpha-1}\|f\|_{q-\ee}^{q-\ee}\,d\ee\\
&\ds{}\leqslant C_4\left[1+\medint_\Omega |f|^q\log^{-\alpha}(a+|f|)\,dx\right]
\end{array}
\end{equation}
for some positive contants.\par
Assume now that $f$ satisfies \eqref{finitezza}. As
\[\|f\|_{q-\ee}^{q-\ee}\leqslant \|f\|_{q-\ee}^{q}+1\]
we see that the first term of \eqref{201309112} is finite, so $f\in \leb^q\log^{-\alpha}\leb(\Omega)$. Furthermore, if $\|f\|_{\leb^q\log^{-\alpha}\leb}=1$, then \eqref{201309112} implies
\[\medint_\Omega |f|^q\log^{-\alpha}(a+|f|)\,dx\le C_{5}\]
for a constant independet of $f$. By homogeneity,
\begin{equation}\label{201309114}
[f]_{\leb^q\log^{-\alpha}\leb}\leqslant C_{5}\, \|f\|_{\leb^q\log^{-\alpha}\leb}
\end{equation}
for all $f$.\par
In case $f\in \leb^q\log^{-\alpha}\leb(\Omega)$, since the Zygmund space is continuously embedded in the gran Lebesgue space $\leb^{\alpha,q)}$ (see \cite{IS}), there exists a constant $C_{6}>0$ such that
\[\|f\|_{q-\ee}\leqslant C_{6}\,\ee^{-\alpha/q}\,[f]_{\leb^q\log^{-\alpha}\leb}\,,\]
thus
\[\|f\|_{q-\ee}^q=\|f\|_{q-\ee}^{q-\ee}\,\|f\|_{q-\ee}^\ee\leqslant\|f\|_{q-\ee}^{q-\ee}\,C_{7}[f]_{\leb^q\log^{-\alpha}\leb}^\ee\]
and by \eqref{201309112} we get \eqref{finitezza}. Infact, if $[f]_{\leb^q\log^{-\alpha}\leb}=1$, then we have
\[\|f\|_{\leb^q\log^{-\alpha}\leb}\leqslant C_{8}\]
and by homogeneity we conclude with the reverse inequality to \eqref{201309114}.
\end{proof}
\begin{remark}
We examine the dependence of $\|~\|_{\leb^q\log^{-\alpha}\leb}$ defined by \eqref{norma equiv}, on the parameter $\ee_0$. For fixed $0<\ee_{0}\leqslant\ee_{1}\leqslant q-1$, by H\"older's inequality we have
\[\|f\|_{q-\ee}\leqslant\|f\|_{q-\ee\,\ee_{0}/\ee_{1}}\,,\]
and hence
\begin{equation}\label{201309191}
\int_{0}^{\ee_{0}}\ee^{\alpha-1}\|f\|_{q-\ee}^q\,d\ee\leqslant \int_{0}^{\ee_{1}}\ee^{\alpha-1}\|f\|_{q-\ee}^q\,d\ee\leqslant \left(\frac{\ee_{1}}{\ee_{0}}\right)^{\alpha}\int_{0}^{\ee_{0}}\ee^{\alpha-1}\|f\|_{q-\ee}^q\,d\ee\,. 
\end{equation}
\end{remark}
\begin{remark}
It is clear that \eqref{finitezza} implies $f\in \leb^{\alpha,q)}(\Omega)$. We remark that the norm \eqref{norma equiv} compares in a very simple way with $\|f\|_{\leb^{\alpha,q)}}$. Indeed, as $\ee\mapsto \|f\|_{q-\ee}$ is decreasing, for all $\sigma\in{}]0,q-1]$ we have
\begin{equation}\label{201309181}
\left\{\int_0^\sigma\ee^{\alpha-1}\|f\|_{q-\ee}^q\,d\ee\right\}^{1/q}\geqslant \|f\|_{q-\sigma}\,\left(\frac {\sigma^\alpha}\alpha\right)^{1/q},
\end{equation}
hence by \eqref{201309191}
\begin{equation}\label{immgranleb}
\|f\|_{\leb^{\alpha,q)}}\leqslant \left(\frac{q-1}{\ee_{0}}\right)^{\alpha/q}\alpha^{1/q}\,\|f\|_{\leb^q\log^{-\alpha}\leb}\,.
\end{equation}
Moreover, using \eqref{norma equiv}, the inclusion $\leb^{\alpha,q)}(\Omega)\subset \leb^q\log^{-\beta}(\Omega)$ for $\beta>\alpha$ (see \cite{Gr}) is trivial:
\[\int_0^{\ee_0}\ee^{\beta-1}\|f\|_{q-\ee}^q\,d\ee=\int_0^{\ee_0}\ee^{\alpha}\|f\|_{q-\ee}^q\,\ee^{\beta-\alpha-1}\,d\ee\]
and then
\[
\|f\|_{\leb^q\log^{-\beta}\leb}\le \left(\frac{\ee_0^{\beta-\alpha}}{\beta-\alpha}\right)^{1/q}\|f\|_{\leb^{\alpha,q)}}\,.\]
\end{remark}
We point out that a
simple application of the Lebesgue dominated convergence theorem proves that
\begin{equation}\label{limit0}
\lim_{\ee\downarrow0}\ee^{\alpha/q}\|f\|_{q-\ee}=0\,,
\end{equation}
for all $f\in \leb^q\log^{-\alpha}\leb(\Omega)$, 
 see \cite{Gr}. Actually, \eqref {limit0} follows directly from \eqref{finitezza}, since it implies that the left hand side of \eqref{201309181} tends to $0$ as $\sigma\downarrow0$.\par
We stress that \eqref{limit0} does not hold uniformly, as $f$ varies in a bounded set of $\leb^q\log^{-\alpha}\leb(\Omega)$. Indeed, for each $\ee>0$ sufficiently small so that $\Phi(\E^{1/\ee})>1$, we choose a measurable subset $E\subset \Omega$ verifying 
($\Omega$ has no atoms)
\[|E|=|\Omega|\,\E^{-q/\ee}\log^\alpha(a+\E^{1/\ee})=|\Omega|/\Phi(\E^{1/\ee})\]
and set
\[f=f_\ee=\E^{1/\ee}\chi_E\,.\]
Then we find that $[f]_{\leb^q\log^{-\alpha}\leb}\equiv1$, while
\[
\begin{array}{rl}
\|f\|_{q-\ee}\kern-.7em&{}=\E^{1/\ee}\E^{-1/\ee\cdot q/(q-\ee)}\log^{\alpha/(q-\ee)}(a+\E^{1/\ee})\\
&{}=\E^{-1/(q-\ee)}\log^{\alpha/(q-\ee)}(a+\E^{1/\ee})
\end{array}
\]
and
\[\lim_{\ee\downarrow0}\ee^{\alpha/q}\|f\|_{q-\ee}=\E^{-1/q}\,.\]
\begin{lemma}\label{relcomp}
For each relatively compact subset $M\subset \leb^q\log^{-\alpha}\leb(\Omega)$, condition \eqref{limit0} holds uniformly for $f\in M$, that is
\[\lim_{\ee\downarrow 0}\left(\sup_{f\in M}\ee^{\alpha/q}\|f\|_{q-\ee}\right)=0\,.\]
\end{lemma}
\begin{proof}
For simplicity, we assume $\ee_{0}=q-1$. As $M$ is totally bounded, fixed arbitrarily $\sigma>0$ we find a finite number of elements $f_1\,,\ldots,f_k\in M$ with the property that, $\forall f\in M$, $\exists j\in\{1\,,\ldots,k\}$:
\[\ee^{\alpha/q}\|f-f_j\|_{q-\ee}\leqslant \|f-f_j\|_{\leb^{\alpha,q)}}\leqslant \alpha^{1/q}\|f-f_j\|_{\leb^q\log^{-\alpha}\leb}<\sigma\,,\]
for all $\ee\in{}]0,\ee_0]$. Above, we used \eqref{immgranleb}. Moreover, $\exists \ee_\sigma\in{}]0,\ee_0]$ such that
\[\ee^{\alpha/q}\|f_j\|_{q-\ee}<\sigma\,,\qquad \forall \ee\in{}]0,\ee_\sigma[\,,\ \forall j\in\{1\,,\ldots,k\}\,.\]
Therefore, we conclude easily for any $f\in M$ and $\ee\in{}]0,\ee_\sigma[$
\[\ee^{\alpha/q}\|f\|_{q-\ee}\leqslant \ee^{\alpha/q}(\|f_j\|_{q-\ee}+\|f-f_j\|_{q-\ee})<2\sigma\,.\]
\end{proof}
In particular, if $(f_n)_{n\in\naturale}$ is a conveging sequence in $\leb^q\log^{-\alpha}\leb(\Omega)$, then
\[\lim_{\ee\downarrow 0}\left(\sup_n\ee^{\alpha/q}\|f_n\|_{q-\ee}\right)=0\,.\]
\subsection{Sobolev space $\sob_0(\Omega)$}
For a bounded domain $\Omega\subset\Rz^N$, let $\sob_0(\Omega)$ be the completion of $C_0^\infty(\Omega)$ with respect to the norm
\[\|u\|=\|\nabla u\|_{\leb^q\log^{-\alpha}\leb}\,.\]

\medskip

\section{Proof of Theorem \ref{main main} \wred{and Theorem \ref{main main 2intro}}}
Our goal in this section is to prove Theorem 
\ref{main main}.
\subsection{A fundamental lemma}
Assume that $\mathcal A=\mathcal A(x,\xi)$
satisfies \eqref{1.2}--\eqref{1.4}.
For $\psi \in \ssob^{1,1}(\Omega)$, we consider  the \wred{equations} 
\wred{
\begin{align}
\label{A harm sec 3}
\divergenza \mathcal A (x,\nabla u)   & = \divergenza f  	\qquad\text{in $  \Omega$},
\\
\label{A harm sec 3 bis}
\divergenza \mathcal A (x,\nabla v)   &  = \divergenza g 	\qquad\text{in $  \Omega$},
\end{align}
}
with $f,g \in \leb^{q-\varepsilon q} (\Omega,\mathbb R^n)$, $0<\varepsilon<1$.
Let \wred{$u,v \in \ssob ^{1,p-\varepsilon p}(\Omega)$} be solutions
to
\eqref{A harm sec 3} and \eqref{A harm sec 3 bis}
respectively
\wred{ 
such that 
\[
u- v \in \ssob_0^{1,p-\varepsilon p}(\Omega)
\]
}
Then
\begin{lemma}\label{uniform lemma}
There exists 
$0 < \varepsilon _ p   (n) < 
1/p$
\wred{
and
a constant $C>0$ depending on $n, p, \alpha, a$ and $b$
}
 such that
the
following uniform estimate holds
\begin{equation}\label{3.3}
\|
\nabla u -\nabla v
\|^p_{p-\varepsilon p}
\leqslant
C
\left(
\varepsilon ^{\frac p {p-2} }
\left\| 
|\nabla u|+|\nabla v|
\right\|^p_{p-\varepsilon p}
+
\left\| 
|f - g|
\right\|^q_{q-\varepsilon q}
\right),
\end{equation}
for every $0<\varepsilon < \varepsilon_p(n)$.
\end{lemma}
\begin{proof}[Proof of Lemma \ref{uniform lemma}]
The proof is 
achieved with a similar argument as in
\cite{GIS}.
We sketch it for the  sake of completeness.

\medskip

Since $\Omega $ is Lipschitz, we may use the 
Hodge decomposition of the vector field 
$|\nabla u - \nabla v|^{-\varepsilon p} (\nabla u - \nabla v) \in 
\leb^{ \frac{p-\varepsilon p}{1-\varepsilon p} }
(\Omega)$
(see \cite{IS,IS2}), namely
\begin{equation}
 	|\nabla u - \nabla v|^{-\varepsilon p} (\nabla u - \nabla v) 
=
\nabla \varphi + h,
\end{equation} 
for some 
$\varphi \in \ssob^{1,{\frac{p-\varepsilon p}{1-\varepsilon p}}}_0(\Omega)$
and some divergence free vector field
$h \in 
\leb^\frac{p-\varepsilon p}{1-\varepsilon p}
(\Omega)
$.
Moreover, 
\wred{
fixed
$0 < \varepsilon _ p   (n) < 
1/p$, 
for every 
$0 <\varepsilon < \varepsilon _ p   (n) $
}
the following estimates hold
(see \cite{IS2})
\begin{align}
\label{Hodge1}
\left\|  {\nabla \varphi} \right\| _  {\frac{p-\varepsilon p}{1-\varepsilon p}}
& \leqslant C(n,p)
\left\|   {\nabla u - \nabla v} \right\|^{1-\varepsilon p} _    {p-\varepsilon p} 
\\
\label{Hodge2}
\left\|  h \right\| _  {\frac{p-\varepsilon p}{1-\varepsilon p}}
& \leqslant C(n,p)
\varepsilon 
\left\|  {\nabla u - \nabla v} \right| ^{1-\varepsilon p} _    {p-\varepsilon p} 
\end{align}
From 
condition \eqref{1.4}
we obtain
\begin{equation}
\begin{split}
\left\|
\nabla u -\nabla v
\right\|^{p-\varepsilon p}_{p-\varepsilon p}
&\leqslant
\average 
\Omega
{\left( |\nabla u|+|\nabla v| \right)^{p-2} |\nabla u - \nabla v|^2 |\nabla u - \nabla v|^{-\varepsilon p}  }
x
\\
&\leqslant
\frac
{1}
{a}
\average
\Omega
{\scal {\mathcal A (x,\nabla u) - \mathcal A (x,\nabla v)}  {|\nabla u - \nabla v|^{-\varepsilon p} \left(\nabla u - \nabla v\right) } }
x
\end{split}
\end{equation}
By Definition \ref{Def weak},
we are legitimate to use $\varphi$ as a test function for 
equations 
in both 
\eqref{A harm sec 3}
and
\eqref{A harm sec 3 bis}
respectively.
Then
\begin{equation} \label{interm11}
\begin{split}
\left\| 
{\nabla u - \nabla v}
\right\|
_{ p-\varepsilon p } ^{p-\varepsilon p}
\leqslant
\frac
{1}
{a}
\biggr[
&
\average
\Omega
{\scal {f-g} {\nabla \varphi} }
x
\\
&
+
\average 
\Omega
{\scal {\mathcal A (x,\nabla u) - \mathcal A (x,\nabla v)}  {h} }
x
\biggr]
\end{split}
\end{equation} 
With the aid of condition \eqref{1.3}
and the
H\"older's inequality, we get
\begin{equation}
\begin{split}
\left\|
\nabla u - \nabla v
\right\|^{p-\varepsilon p}_{p-\varepsilon p}
\leqslant
\frac
{1}
{a}
\biggr[
&
\|
f-g
\|_{q-\varepsilon q}
\|
\nabla \varphi  
\|_{\frac {p-\varepsilon p} {1-\varepsilon p} }
\\
&
+
b
\|
|\nabla u|
+
|\nabla v|
\|^{p-2}_{p-\varepsilon p}
\|
|\nabla u
-
\nabla v|
\|_{p-\varepsilon p}
\|
h
\|_{\frac {p-\varepsilon p} {1-\varepsilon p} }
\biggr]
\end{split}
\end{equation}
which, in view of 
\eqref{Hodge1}
and
\eqref{Hodge2},
yields
\begin{equation}
\begin{split}
\left\|
\nabla u - \nabla v
\right\|^{p-\varepsilon p}_{p-\varepsilon p}
\leqslant
C
\biggr[
&
\|
f-g
\|_{q-\varepsilon q}
\|
  \nabla u - \nabla v
\|^{ 1-\varepsilon p } _{ p-\varepsilon p }  
\\
&
+
\varepsilon
\|
|\nabla u|
+
|\nabla v|
\|_{p-\varepsilon p}
^{p-2}
\|
|\nabla u
-
\nabla v|
\|^{2-\varepsilon p}_{p-\varepsilon p}
\biggr]
\end{split}
\end{equation}
where $C=C(n,p,a,b)$. 
With the aid of Young's inequality we obtain
\begin{equation}
\begin{split}
\left\|
\nabla u - \nabla v
\right\|^{p-1}   _{p-\varepsilon p}
&
\leqslant
C
\|
f
-
g
\|_{q-\varepsilon q}
+
C
\varepsilon
\|
|\nabla u|
+
|\nabla v|
\|_{p-\varepsilon p}
^{p-2}
\left\|
\nabla u - \nabla v
\right\|_{p-\varepsilon p}
\\
&
\leqslant
C
\left(
\|
f
-
g
\|_{q-\varepsilon q}
+
\varepsilon ^ {\frac {p-1} {p-2} }
\|
|\nabla u|
+
|\nabla v|
\|_{p-\varepsilon p}
^{p-1}
\right) 
\\
&
\quad
+ 
\frac 1 {(p-1) 2^{p-1}}    
\left\|
\nabla u - \nabla v
\right\|^{p-1}_{p-\varepsilon p}
\end{split}
\end{equation}
Once the latter term is absorbed by the left hand side, we have
\begin{equation}\label{3.12}
\begin{split}
\left\|
\nabla u - \nabla v
\right\|^{p-1}   _{p-\varepsilon p}
&
\leqslant
C
\left(
\|
f
-
g
\|_{q-\varepsilon q}
+
\varepsilon ^ {\frac {p-1} {p-2} }
\|
|\nabla u|
+
|\nabla v|
\|_{p-\varepsilon p}
^{p-1}
\right)  
\end{split}
\end{equation}
which corresponds to the estimate we wanted to prove.
\end{proof}
\begin{corollary}
Under the assuptions of Lemma 3.1, if \wred{$u=v$ on $\partial \Omega$}, 
there exists \wred{\wred{$0<\varepsilon_0<1/p$} and
a constant $C>0$ depending on $n, p, \alpha, a$ and $b$} 
such that,
for any $0<\varepsilon<\varepsilon_0$
the
following uniform estimate holds
\begin{equation}\label{3.3bis}
\|
\nabla u -\nabla v
\|^p_{p-\varepsilon p}
\leqslant
C
\left(
\varepsilon ^{\frac p {p-2} }
\left\| 
|f|+|g|
\right\|^q_{q-\varepsilon q}
+
\left\| 
f - g 
\right\|^q_{q-\varepsilon q}
\right),
\end{equation}
\end{corollary}
\begin{proof}
For $g=0$ and $v=0$,
estimate \eqref{3.3}
reduces to
\begin{equation}\label{3.13}
\begin{split}
\left\|
\nabla u  
\right\|^{p-1}   _{p-\varepsilon p}
&
\leqslant
C
\left(
\|
f
\|_{q-\varepsilon q}
+
\varepsilon ^ {\frac {p-1} {p-2} }
\| 
\nabla u 
\|_{p-\varepsilon p}
^{p-1}
\right)  
\end{split}
\end{equation}
which gives, for  $C \varepsilon^{\frac {p-1} {p-2} } <1$ 
\begin{equation}\label{3.14 bis}
\begin{split}
\left\|
\nabla u  
\right\|^{p-1}   _{p-\varepsilon p}
&
\leqslant
C
\|
f
\|_{q-\varepsilon q} ^{q-1}
\end{split}
\end{equation}
Similarly, one has
\begin{equation}\label{3.14 ter}
\begin{split}
\left\|
\nabla v  
\right\|^{p-1}   _{p-\varepsilon p}
&
\leqslant
C
\|
g
\|_{q-\varepsilon q} ^{q-1}
\end{split}
\end{equation}
Inserting 
\eqref{3.14 bis}
and
\eqref{3.14 ter}
into
\eqref{3.12}, we finally get \eqref{3.3bis}.
\end{proof}

\subsection{Uniqueness}
Under the assumptions of Theorem \ref{main main},
if $f=g$, estimate \eqref{3.3bis}    reduces to
\begin{equation}\label{3.11}
\|
\nabla u -\nabla v
\|^p_{p-\varepsilon p}
\leqslant
C
\varepsilon ^{\frac p {p-2} }
\left\| 
 f 
\right\|^q_{q-\varepsilon q}
.
\end{equation}
Then, if $f \in \leb^q \log^{-\alpha} \leb (\Omega,\mathbb R^n)$, $0 < \alpha \leqslant p/(p-2)$,
uniqueness follows from \eqref{limit0}
letting $\varepsilon \rightarrow 0^+$ in \eqref{3.11}.
Actually, we can prove a stronger uniqueness result.

\begin{theorem}\label{main main 2}
Assume \eqref{1.2}--\eqref{1.4} hold.
There exist $ s \in (p-1/p,p)$ \wred{depending only on $n, p,  a$ and  $b$}, such that 
if $u,v \in \ssob^{1,1}(\Omega)$
satisfy 
$u-v \in \ssob_0^{1,1}(\Omega)$,
$\nabla u \in \leb ^p \log^{-\alpha} \leb (\Omega,\mathbb R^n)$, 
$0<\alpha \leqslant   p/(p-2)$, $\nabla v \in \leb^s(\Omega,  \mathbb R^n)$ and
\begin{equation}
\divergenza \mathcal A (x,\nabla u)
= 
\divergenza \mathcal A (x,\nabla v)
\end{equation}
then $u=v$ in $\Omega$.
\end{theorem}

\begin{proof}[Proof of Theorem \ref{main main 2}]
Arguing 
as
in Lemma \ref{uniform lemma},
we decompose
the vector field 
$|\nabla u - \nabla v|^{-\varepsilon p} (\nabla u - \nabla v) \in 
\leb^{ \frac{p-\varepsilon p}{1-\varepsilon p} }
(\Omega)$
and  
for $f=g$ we get the following estimate
\begin{equation}
\left\|
\nabla u - \nabla v
\right\|_{p-\varepsilon p} ^p
\le 
C\varepsilon ^ {\frac p { p-2} }
\left\|
|\nabla u| + |\nabla v|
\right\|_{p-\varepsilon p}  ^p
\end{equation}
which yields
\begin{equation}\label{3.14}
\left\|
\nabla u - \nabla v
\right\|_{p-\varepsilon p} ^p
\le 
C\varepsilon ^ { \frac  {p} {p-2}  }
\left(
\left\|
\nabla u - \nabla v 
\right\|_{p-\varepsilon p} ^ {p}
+
\left\|
\nabla u   
\right\|_{p-\varepsilon p} ^ {p}
\right)
\end{equation}
for $0 < \varepsilon <\varepsilon_p(n) $
and
$C=C(n,p, a ,b)$.
Now, if 
$0 < \varepsilon < \min \left\{ \varepsilon_p(n), 1/C^{\frac {p-2}{p}}  \right\} $, the first term in the right hand side can be absorbed by the left hand side of \eqref{3.14} and so
\begin{equation}\label{3.15}
\left\|
\nabla u - \nabla v
\right\|_{p-\varepsilon p} ^p
\le 
\left(
\frac
{C\varepsilon}
{1 - C \varepsilon}
\right)
^{ \frac p {p-2} }
\left\|
\nabla u   
\right\|_{p-\varepsilon p}   ^p
\end{equation}
The conclusion of Theorem \ref{main main 2} follows by \eqref{limit0}, as $\varepsilon \rightarrow 0^+$ in \eqref{3.15}.
\end{proof}

The previous theorem improves  the uniqueness result of
\cite{GIS}, which does not cover the case $\alpha=  p/(p-2)$.
We point ou that  
our result also improves
the result in \cite{DHM}, since the 
Marcinkiewicz 
space $weak-\leb^p$ is contained in $\leb^p \log^{-\alpha} \leb $ when $1<\alpha \le p/(p-2)$.
Actually, estimate  \eqref{3.3}   
allows us to give a simple proof of \cite[Theorem 4.2]{DHM}.  
Arguing as in Theorem 
\ref{main main 2}
we arrive at 
\eqref{3.14}
for 
$|\nabla u|
\in
\leb^{p,\infty} (\Omega)
$
and 
$|\nabla v|
\in
\leb^{s} (\Omega)
$.
Then, by H\"older's inequality
\eqref{Holder Mar}
we get
\begin{equation}\label{Mar 1}
\left\|
\nabla u - \nabla v
\right\|_{p-\varepsilon p}
^p
\leqslant    
C 
\varepsilon^{\frac p {p-2} -1 }
\|
\nabla u
\|^p_{p,\infty}
\end{equation} 
and 
letting $\varepsilon \rightarrow 0^+$
we have $u=v$ in $\Omega$.
\subsection{Existence}
Let 
$f\in \leb^q\log^{-\alpha}\leb(\Omega,\Rz^N)$, 
$1<q<2$ and $0<\alpha\leqslant p/(p-2)$.
The aim of this subsection
is to prove the existence in Theorem \ref{main main}.
\wred{As a preliminary step, we show}
that, if $(f_n)_n$ is a converging sequence in $\leb^q\log^{-\alpha}\leb(\Omega,\Rz^N)$, such that 
for each $n$
\begin{equation}\label{pb approssimanti}
\left\{
\begin{array}l
\divergenza
\mathcal A (x,  \nabla u_n )  =\divergenza f_n\\
\wred{u_n = 0 \quad \text{on $\partial \Omega$} }  
\end{array}
\right.
\end{equation}
then $(\nabla u_n)_n$ is a Cauchy sequence in $\leb^p\log^{-\alpha}\leb(\Omega,\Rz^N)$. 
To prove this, we first note that, by Lemma~\ref{relcomp}, |wred{fixed} $  \sigma>0$, we find $\vartheta\in{}]0,1]$ such that, if $0<\ee<\vartheta \ee_p(n)$, then
\[\ee^{\alpha}\||f_m|+|f_n|\|_{q-\ee q}^q<\sigma\,,\]
for all $m,n\in\naturale$. Hence
%
%
%
\wred{\eqref{3.3bis}}
with $f_{m}$, $f_{n}$ in place of $f$, $g$, and $u_{m}$, $u_{n}$ in place of $u$, $v$, respectively, yields
\begin{equation}\label{stimaGIS2}
\|\nabla u_m-\nabla u_n\|_{p-\ee p}^p\lesssim \sigma+\|f_m-f_n\|_{q-\ee q}^q\,.
\end{equation}
We multiply both sides by $\ee^{\alpha-1}$ and integrate with respect to $\ee$ on $(0,\vartheta\ee_p(n))$. For $\delta=\ee p/\vartheta\geqslant\ee p$, we have
\[\|\nabla u_m-\nabla u_n\|_{p-\ee p}\geqslant \|\nabla u_m-\nabla u_n\|_{p-\delta}\,,\]
hence
\begin{equation}\label{201309182}
\int_0^{\vartheta\ee_p(n)}\ee^{\alpha-1}\|\nabla u_m-\nabla u_n\|_{p-\ee p}^p\,d\ee\geqslant 
\left(\frac \vartheta p\right)^\alpha\!\int_0^{\ee_0}\delta^{\alpha-1}\|\nabla u_m-\nabla u_n\|_{p-\delta}^p\,d\delta\,,
\end{equation}
where $\ee_0=p\ee_p(n)$. On the other hand,
\[\int_0^{\vartheta\ee_p(n)}\ee^{\alpha-1}\,d\ee=\frac{(\vartheta\ee_p(n))^\alpha}\alpha\]
and (setting here $\delta=\ee q$)
\begin{equation}\label{201309183}
\int_0^{\vartheta\ee_p(n)}\ee^{\alpha-1}\|f_m-f_n\|_{q-\ee q}^q\,d\ee\leqslant 
q^{-\alpha}\int_0^{\ee_0}\delta^{\alpha-1}\|f_m-f_n\|_{q-\delta}^q\,d\delta\,.
\end{equation}
Therefore, recalling definition \eqref{norma equiv}, from \eqref{stimaGIS2} we get
\begin{equation}\label{stimaGIS3}
\|\nabla u_m-\nabla u_n\|_{\leb^p\log^{-\alpha}\leb}^p\lesssim \sigma+\vartheta^{-\alpha}\|f_m-f_n\|_{\leb^q\log^{-\alpha}\leb}^q
\end{equation}
with no restrictions on $m,n\in\naturale$. Now, as the sequence $(f_n)_n$ conveges in $\leb^q\log^{-\alpha}\leb(\Omega)$, we have
\[\vartheta^{-\alpha}\|f_m-f_n\|_{\leb^q\log^{-\alpha}\leb}^q<\sigma\,,\]
provided $m$ and $n$ are sufficiently large, hence
\[\|\nabla u_m-\nabla u_n\|_{\leb^p\log^{-\alpha}\leb}^p\lesssim \sigma\]
proving that $(\nabla u_n)_n$ is a Cauchy sequence as desired.\par

\medskip

Now we are in a position to prove existence of solution for problem \eqref{A harm intro f}. Indeed, we approximate the vector field $f$ in the right hand side of the equation by $f_n\in \leb^q(\Omega,\Rz^N)$, $n=1\,,2\,,\ldots$, such that $f_n\to f$ in $\leb^q\log^{-\alpha}\leb(\Omega,\Rz^N)$, and for each $n$ we consider the (unique) solution $u_n$ to the problem
\begin{equation}\label{pb approssimanti2}
\left\{
\begin{array}l
\divergenza  \mathcal A (x,  \nabla u_n )    =\divergenza f_n\\
u_n\in \ssob^{1,p}_0(\Omega)
\end{array}
\right.
\end{equation}
Using what we have seen above, $(u_n)_n$ converges in $\sob^p_0(\Omega)$, that is, there exists $u\in \sob^p_0(\Omega)$ such that $u_n\to u$. To conclude that $u$ solves
\eqref{A harm intro f}, 
we only need to note, 
that
by
\eqref{1.3}
 we can pass to the limit as $n\to \infty$ into the equation of \eqref{pb approssimanti2}, getting
\[\divergenza     \mathcal A (x,  \nabla u )     =\divergenza f\,,\]
since  %
$\nabla u_n\to\nabla u$ in $\leb^{p-1}(\Omega,\Rz^N)$ in particular.\par
The 
estimate
\wred{\eqref{tag eqn}}
follows from \eqref{3.14 bis}, \wred{by} the same argument used above, by integrating with respect to $\varepsilon$.\par
\wred
{
Also continuity of the operator $\mathcal H$ follows. Indeed, clearly $f_n \rightarrow f$ in 
$\leb ^ q \log^{-\alpha} \leb $ implies $\nabla u_n=\mathcal H f_n \rightarrow \nabla u = \mathcal H$ in  $\leb ^ p \log^{-\alpha} \leb $. \qed
}

\begin{proof}[Proof of Theorem \ref{main main 2intro}]

Let now $0<\alpha<p/(p-2)$
and let
$f,g \in \leb^q \log ^{-\alpha} \leb (\Omega,\mathbb R^n)$. Denote by $u$ and $v$ the solutions of 
\eqref{A harm sec 3} and 
\eqref{A harm sec 3 bis}, of class $\ssob^1 \leb ^p \log ^{-\frac p {p-2}} \leb (\Omega)$, respectively.
To prove \eqref{3.3 ter},
we multiply both sides of 
\eqref{3.3bis}
by $\ee^{\alpha-1}$ and integrate with respect to $\ee$ on $(0,\vartheta\ee_p(n))$, for fixed $\vartheta\in{}]0,1]$. Similarly as for \eqref{201309182} and \eqref{201309183}, we have
\begin{equation}\label{201309185}
\int_0^{\vartheta\ee_p(n)}\ee^{\alpha-1}\|\nabla u-\nabla v\|_{p-\ee p}^p\,d\ee\geqslant 
\left(\frac \vartheta p\right)^\alpha\!\int_0^{\ee_0}\delta^{\alpha-1}\|\nabla u-\nabla v\|_{p-\delta}^p\,d\delta\,,
\end{equation}
\begin{equation}\label{201309186}
\int_0^{\vartheta\ee_p(n)}\ee^{\alpha-1}\|f-g\|_{q-\ee q}^q\,d\ee\leqslant 
q^{-\alpha}\int_0^{\ee_0}\delta^{\alpha-1}\|f-g\|_{q-\delta}^q\,d\delta\,.
\end{equation}
respectively. On the other hand,
\begin{equation}\label{201309184}
\int_0^{\vartheta\ee_p(n)}\ee^{\frac p{p-2}+\alpha-1}\||f|+|g|\|_{q-\ee q}^q\,d\ee\leqslant 
\frac{(\vartheta\ee_p(n))^{\frac p{p-2}}}{q^{\alpha}}
\int_0^{\ee_0}\delta^{\alpha-1}\||f|+|g|\|_{q-\delta}^q\,d\delta
\end{equation}
and therefore we get
\begin{equation}\label{stimaGIS5}
\|\nabla u-\nabla v\|_{\leb^p\log^{-\alpha}\leb}^p\lesssim \vartheta^{\frac p{p-2}-\alpha}\||f|+|g|\|_{\leb^q\log^{-\alpha}\leb}^q+\vartheta^{-\alpha}\|f-g\|_{\leb^q\log^{-\alpha}\leb}^q
\end{equation}
For
\[\vartheta^{\frac p{p-2}}=\frac{\|f-g\|_{\leb^q\log^{-\alpha}\leb}^q}{\||f|+|g|\|_{\leb^q\log^{-\alpha}\leb}^q}\]
we obtain estimate \eqref{3.3 ter}. 
In particular, for $g=0$ and $v=0$,
\begin{equation}\label{201309201}
\|\nabla u\|_{\leb^p\log^{-\alpha}\leb}^p\lesssim 
\|f\|_{\leb^q\log^{-\alpha}\leb}^q\,.
\end{equation}
\end{proof}

\begin{remark}
Assume $f$ and $g$ in $\leb^q\log^{-p/(p-2)}\leb(\Omega,\Rz^N)$, and let $u$ and $v$ in $\Sob{p/(p-2)}_0(\Omega)$ solve 
\wred
{
\eqref{A harm sec 3} and 
\eqref{A harm sec 3 bis} respectively.
}
For $0\leqslant \alpha<p/(p-2)$, we can prove that
\[f-g\in \leb^q\log^{-\alpha}\leb(\Omega,\Rz^N)\implica \nabla u-\nabla v\in \leb^p\log^{-\alpha}\leb(\Omega,\Rz^N)\,.\]
Indeed, in the case $\alpha=0$, passing to the limit as $\ee\downarrow 0$ in \wred{\eqref{3.3}}, 
we find that $\nabla u-\nabla v\in \leb^p(\Omega,\Rz^N)$ and
\begin{equation}\label{stimaGIS4}
\|\nabla u-\nabla v\|_p^p\lesssim \|f-g\|_q^q\,.
\end{equation}
In the case $0<\alpha<p/(p-2)$, similarly as for \eqref{201309184} we find ($\vartheta=1$)
\begin{equation}\label{201309187}
\int_0^{ \ee_p(n)}\ee^{\frac p{p-2}+\alpha-1}\||f|+|g|\|_{q-\ee q}^q\,d\ee\leqslant 
\frac{(\ee_p(n))^{\alpha}}{q^{\frac p{p-2}}}
\int_0^{\ee_0}\delta^{\frac p{p-2}-1}\||f|+|g|\|_{q-\delta}^q\,d\delta\,.
\end{equation}
By \eqref{201309185}, \eqref{201309186} and \eqref{201309187} we get
\begin{equation}\label{201309188}
\|\nabla u-\nabla v\|_{\leb^p\log^{-\alpha}\leb}^p\lesssim \||f|+|g|\|_{\leb^q\log^{-\frac p{p-2}}\leb}^q+\|f-g\|_{\leb^q\log^{-\alpha}\leb}^q
\end{equation}
\end{remark}

\begin{proof}[Proof of Theorem \ref{comparison}] 
Under assumption 
\eqref{unif ell}
it is easy to verify that 
$\mathcal A(x,\xi)$ defined in \eqref{matric matcal A}
satisfies assuptions \eqref{1.2}--\eqref{1.4} with $\lambda=a$.
By arguing as in the proof of Lemma 3.1, as in \cite{FM} we get
\begin{equation}\label{4.1}
\|
\nabla u - \nabla v
\|_{p-\varepsilon p}^p
\leqslant 
C(n,p,a,b)
\left\{
\left(K_A-1\right)^{\frac p {p-1}} \| \nabla v \|^p_{p-\varepsilon p} + \varepsilon ^{\frac{p}{p-2}}
\left\|  |\nabla u|+|\nabla v| \right\|_{p-\varepsilon p}^p 
\right\}
\end{equation}
which holds true as long as $\varepsilon \in (0,\varepsilon_p(n)) $ for some $\varepsilon_p(n)>0$.
Let us fix some $\vartheta\in(0,1)$ which will be properly choosen later.
Let us consider the integrals
\begin{equation}
\begin{split}
I_1 & =\int_0^{\vartheta \varepsilon_p(n)} \varepsilon ^ {\alpha-1} \| \nabla u -\nabla v \|^p_{p-\varepsilon p} d\varepsilon
\\
I_2 & =\int_0^{\vartheta \varepsilon_p(n)} \varepsilon ^ {\alpha-1} \|  \nabla v \|^p_{p-\varepsilon p} d\varepsilon
\\
I_3 & =\int_0^{\vartheta \varepsilon_p(n)} \varepsilon ^ {\frac p {p-2} +\alpha-1} \left\| |\nabla u| + |\nabla v| \right\|^p_{p-\varepsilon p} d\varepsilon
\end{split}
\end{equation}
so that estimate \eqref{4.1}
infers
\begin{equation}\label{4.1 bis}
I_1 \leqslant C(n,p,a,b) 
\left\{
\left(K_A-1\right)^{\frac p {p-1}} I_2 +  I_3
\right\}
\end{equation}
We set 
\[
\delta =\frac {\varepsilon p} \vartheta
\]
Since $\delta  \geqslant   {\varepsilon p}$, a use of Holder's inequality
allow us to obtain
\begin{equation}\label{4.2}
\begin{split}
I_1 
&
\geqslant  
\int_0^{\vartheta \varepsilon_p(n)} \varepsilon ^ {\alpha-1} \| \nabla u -\nabla v \|^p_{p-\delta  } d\varepsilon
\\
&
=
\left(
\frac \vartheta p
\right)^\alpha
\int_0^{p\varepsilon_p(n)}
\delta ^ {\alpha-1} \| \nabla u -\nabla v \|^p_{p-\delta  } d\delta
\end{split}
\end{equation}
On the other hand, since $  0   \leqslant \vartheta \leqslant 1$ we have
\begin{equation}\label{4.3}
I_2 \leqslant  \int_0^{  \varepsilon_p(n)} \varepsilon ^ {\alpha-1} \|  \nabla v \|^p_{p-\varepsilon p} d\varepsilon
=
\frac 1 {p^\alpha}
  \int_0^{ p \varepsilon_p(n)} \delta ^ {\alpha-1} \|  \nabla v \|^p_{p-\delta} d\delta
\end{equation}
Similarly,
\begin{equation}\label{4.4}
I_3 \leqslant \left(   \vartheta \varepsilon_p(n)  \right)^{\frac p {p-2}}  \int_0^{ \theta \varepsilon_p(n)} \varepsilon ^ {\alpha-1} \|   |\nabla u|+|\nabla v| \|^p_{p-\varepsilon p} d\varepsilon
=
C(n,p,\alpha)
  \int_0^{ p \varepsilon_p(n)} \delta ^ {\alpha-1}  \left\| |\nabla u| + |\nabla v| \right\|^p_{p-\delta} d\delta
\end{equation}
Combining 
\eqref{4.2}, \eqref{4.3} and \eqref{4.4} with 
\eqref{4.1}
we have
\begin{equation}\label{4.5}
\vartheta^\alpha 
\|\nabla u - \nabla v \|^p_{\leb^p\log^{-\alpha}\leb}
\leqslant
C
\left\{
\left(K_A-1\right)^{\frac p {p-1}}
\|  \nabla v \|^p_{\leb^p\log^{-\alpha}\leb}
+
\vartheta^{\frac p {p-2}} \left\| |\nabla u| + |\nabla v| \right\|^p_{\leb^p\log^{-\alpha}\leb}
\right\}
\end{equation}
Now, we pick $\vartheta$ in such a way that 
\[
\vartheta^{\frac p {p-2}} = \left( \frac {K_A-1} {K_A} \right)^{\frac p {p-1}} 
\]
Hence, 
\eqref{4.5}
may be rewritten as
\begin{equation}\label{4.6}
\begin{split}
\|\nabla u - \nabla v \|^p_{\leb^p\log^{-\alpha}\leb}
\leqslant
C \left(K_A-1\right)^{\frac p {p-1} - \alpha \frac {p-2} {p-1} }
K_A^{\frac {\alpha(p-2)} {p-1} }
&
\biggr\{
K_A^{\frac p {p-1} }
\|  \nabla v \|^p_{\leb^p\log^{-\alpha}\leb}
\\
&
+
\left\| |\nabla u| + |\nabla v| \right\|^p_{\leb^p\log^{-\alpha}\leb}
\biggr\}
\end{split}
\end{equation}
Finally,   \eqref{4.7} is proved.
\end{proof}

\vskip 0.7cm
{
\tabskip 0pt plus \hsize
\halign to \hsize{\hfil#\tabskip 0pt \cr
Dipartimento di Matematica e Applicazioni ``R.\ Caccioppoli''\cr
Universit\`a degli Studi di Napoli ``Federico II''\cr
Via Cintia -- 80126 {\scshape Napoli}\cr
fernando.farroni@unina.it\cr
luigreco@unina.it\cr
gmoscari@unina.it\cr}}
\end{document}